\documentclass[12pt,a4paper]{amsart}
\usepackage{amssymb,amsmath,amsfonts,pinlabel}
\input xy
\xyoption{all}

\numberwithin{equation}{section} 
\numberwithin{equation}{subsection}

\theoremstyle{plain} 
\newtheorem{theorem}[equation]{Theorem} 
\newtheorem{lemma}[equation]{Lemma} 
\newtheorem{proposition}[equation]{Proposition}

\theoremstyle{definition}

\newtheorem{definition}[equation]{Definition}

\newcommand{\gras}[1]{{\mathbb #1}} 
\newcommand{\C}{\gras{C}} 
\newcommand{\Z}{\gras{Z}} 
 
\newcommand{\N}{\gras{N}}

\begin{document}   
\title[Milnor fibers of sandwiched singularities]{On 
the Milnor fibers of sandwiched   singularities} 
\author{\sc Andr{\'a}s N{\'e}methi} 
\address{Renyi Institute of Mathematics, POB 127, H-1364 Budapest, Hungary.} 
\email{nemethi@renyi.hu} 
\author{\sc Patrick Popescu-Pampu} 
\address{Univ. Paris 7 Denis Diderot, Inst. de 
  Maths.-UMR CNRS 7586, {\'e}quipe "G{\'e}om{\'e}trie et dynamique" \\ 
  Site Chevaleret, Case 
  7012, 75205 Paris Cedex 13, France.} 
\email{ppopescu@math.jussieu.fr}

\subjclass{32S55, 53D10, 32S25, 57R17}

\keywords{Sandwiched singularities, versal deformation, 
 smoothings, Milnor fibers, Stein fillings}

\begin{abstract} The {\em sandwiched} surface singularities  
are those rational surface singularities which dominate birationally 
smooth surface singularities. de Jong and van Straten showed that one 
can reduce the study of the deformations of a 
sandwiched surface singularity to the study of deformations of a 1-dimensional 
object, a so-called {\em decorated} plane curve singularity. In particular, 
the Milnor fibers corresponding to their various smoothing components 
may be reconstructed up to diffeomorphisms from those deformations 
of associated decorated curves which have only ordinary singularities. 
Part of the topology of such a deformation is encoded in the {\em 
incidence matrix} between the irreducible components of the deformed curve 
and the points which decorate it, well-defined up to permutations of
columns. Extending a previous theorem ofours, which treated the case  
of cyclic quotient singularities, we show that the Milnor fibers 
which correspond to  deformations whose incidence matrices are different 
up to permutations of columns are not diffeomorphic in a strong
sense. This gives a lower bound on the
number of Stein fillings of the contact boundary of a sandwiched
singularity. 
  \end{abstract}

\maketitle 
\pagestyle{myheadings} \markboth{{\normalsize 
A. N{\'e}methi and P. Popescu-Pampu}} 
{{\normalsize On  the Milnor fibers of sandwiched singularities}}


\section{Introduction} \label{intro} 
 
In our previous paper \cite{NPP 08}, we proved a conjecture of Lisca
\cite{L 04, L 08},
establishing a bijective correspondence between the Milnor fibers of
the smoothing components of a cyclic quotient singularity and the
Stein fillings of the corresponding contact lens space, the boundary of the
singularity. As a
particular case of our results, the Milnor fibers corresponding to
distinct smoothing components of the reduced miniversal base space of
the cyclic quotient singularity are
pairwise non-diffeomorphic. Here one has to understand the
diffeomorphisms {\em in a strong sense}: namely, there are natural
identifications of the boundaries of the Milnor fibers up to isotopy,
and we showed that there are no diffeomorphisms extending those
identifications. 

Having in mind the construction of de Jong and van Straten in \cite{JS 98}
regarding the Milnor fibers of sandwiched singularites, it is 
natural to try to extend the above result for such singularities. 
One of the main obstructions of this program is that at the present moment the 
precise description of all the smoothing components of sandwiched 
singularities (or/and the classification of all Stein fillings of the 
corresponding contact boundaries) is  out of hope. 

Nevertheless, we will prove a slightly weaker version of the above result. 
In order to explain it we need some preparation. 

The  class  of sandwiched singularities 
received its name in \cite{S 90} from the fact that
the corresponding singularities are analytically isomorphic to the
germs of surfaces which may be ``sandwiched'' birationally between
two smooth surfaces (see also subsection \ref{SAND}). 
They form a subclass of the rational surface
singularities and their study may be reduced in many respects to the
study of plane curve singularities. 

Indeed, de Jong and  van Straten
showed that this was the case for their deformation theory as well. 
Namely, one may encode a sandwiched singularity $(X,x)$ by a {\em
decorated curve} $(C,l)$, which is a germ of reduced plane curve
whose components $C_i$ are decorated by sufficiently high positive 
integers $l_i$.  One of the main
theorems of \cite{JS 98} states that the deformation properties of a
sandwiched singularity reduces to the deformation properties of any such 
decorated curve which encodes it. 

In particular, the 1-parameter smoothings of $(X,x)$ correspond
bijectively to the so-called {\em picture deformations} of $(C,l)$,
which are $\delta$-constant deformations of $C$ with generic fibers
having only ordinary singularities, accompanied with flat 
deformations of $l$ (seen as a subscheme of the normalization of $C$)
which generically are  reduced and contain the preimage of the singular
locus of the generic fiber. The Milnor fiber corresponding to a
smoothing component of $(X,x)$ may be
reconstructed up to strong diffeomorphism from the embedding in a
4-ball of the generic fiber of a corresponding picture deformation.  

Part of the topological structure of the generic fiber of the
deformation of $(C,l)$  may be encoded in an associated {\em incidence
  matrix} between the irreducible components of the deformation of $C$
and the images of the support of $l$ seen as a subscheme of the
normalization. This matrix is well-defined up to permutation of
columns. 

As it is proved in  \cite{NPP 08}, the validity of Lisca's conjecture
implies that, {\em for cyclic 
  quotient singularities},
the Milnor fibers which correspond to distinct
incidence matrices (up to permutation of columns) are not strongly
diffeomorphic. 

The aim of this paper is to show that {\em the same is true for all
sandwiched surface singularities} (see Theorem \ref{distinc}). 

In fact, for any fixed realization of $(X,x)$  by a decorated curve 
$(C,l)$, we provide a canonical method to reproduce all the incidence matrices
(associated with $(C,l)$) from the Milnor fibers. Hence, each realization
$(C,l)$ provides a `test' to separate the Milnor fibers (in fact, precise 
criteria to recognize  them). 

Recall also that de Jong and van Straten predict (see \cite[(4.3)]{JS 98})
that, for any fixed $(C,l)$, the set of smoothing components of the 
deformation space of $(X,x)$ injects into the set of incidence matrices 
associated with $(C,l)$. In particular, for any class of singularities, when 
this is indeed the case, we obtain via our result that there is no pair of 
smoothing components with (strongly) diffeomorphic Milnor fibers. 

In order to connect our main result with the set of Stein fillings of the 
contact boundary, we have to face a new aspect of the problem: sandwiched
singularities are not necessarily taut (that is, their analytical
structure is not determined by their topology). Moreover, fixing a 
topological type, for different
analytic types one might have different structure of the miniversal
deformation space (therefore different sets of Milnor fibers). 
Nevertheless, for a fixed topological type, the induced natural contact
structure on the boundary is independent of the analytic realization and 
is invariant up to isotopy by all orientation-preserving
diffeomorphisms (cf. \cite{CP 04}).

Therefore, in Theorem \ref{lowbound} we are able to extend the above 
recognition criterion of the Milnor
fibers via incidence matrices by considering all the smoothings associated
with the deformations of all the decorated germs of plane curves with a fixed
topology. In this way we get a lower bound for the
number of Stein fillings of the contact boundary of a sandwiched
surface singularity, considered up to orientation-preserving
diffeomorphisms fixed on the boundary.

\subsection{Conventions and notations} \label{conv} 
All the differentiable manifolds we consider are oriented: any
letter, say $W$,  
 denoting a manifold denotes in fact an oriented  manifold. 
We denote by $\overline{W}$ the manifold 
obtained by changing the orientation of $W$, 
and by $\partial W$ its boundary, 
canonically oriented by the rule that the outward 
normal \emph{followed} by the orientation of $\partial W$ gives the 
orientation of $W$. 
 
We work exclusively with homology groups \emph{with 
integral coefficients}. If $W$ is 4-dimen\-sional, the intersection number 
in $H_2(W)$ is denoted by $S_1\cdot S_2$. 
An element of $H_2(W)$ is called \emph{a $(-1)$-class}
if its self-intersection is equal to $-1$.  

If $W_1$ and $W_2$ are two oriented manifolds with boundary, 
endowed with a fixed isotopy class of orientation-preserving diffeomorphisms 
$\partial W_1 \rightarrow \partial W_2$, we say that $W_1$ and $W_2$ 
are {\em strongly diffeomorphic} (with respect to that class) if there 
exists an orientation-preserving diffeomorphism $W_1 \rightarrow W_2$ 
whose restriction to the boundary belongs to the given class.

 \medskip 
\section{Generalities} 
\label{general}

The aim of this section is to recall briefly the needed facts about
deformation theory, $\delta$-constant deformations and the topology of
surface singularities. 
In what follows, a \emph{singularity} will design a germ of reduced
complex analytic space.

\subsection{Deformations and smoothings} \label{defsm}

Let $(X,x)$ be an isolated singularity. Choose one of its representatives
embedded in some affine space: $(X,x) \hookrightarrow (\C^n,0)$. Denote
by $\mathbb{B}_r\subset \C^n$ the compact ball of radius $r$ centered at the
origin and by $\mathbb{S}_r$ its boundary. By a general theorem (see
\cite{L 84}), there exists $r_0 >0$ such that $X$ is transversal to
the euclidean  spheres $\mathbb{S}_r$ for any $r \in (0,
r_0]$. $\mathbb{B}_r$ is then called a \emph{Milnor ball} for the
chosen embedding and such a
representative $X \cap \mathbb{B}_r$ is called a \emph{Milnor
  representative} of $(X,x)$. The manifold $X \cap \mathbb{S}_r$,
oriented as boundary of the complex manifold $(X \cap
\mathbb{B}_r)\setminus \{x\}$ is  well-defined
up to orientation-preserving diffeomorphisms, and is called the
\emph{(abstract) boundary} of $(X,x)$. We denote it by
$\partial(X,x)$. 
 
 \begin{definition} \label{miniv} 
  Let $(X,x)$ be a singularity. A 
  {\it deformation} of $(X,x)$ is a germ of flat morphism 
  $\pi:(Y,y)\rightarrow (S,0)$ together with an 
  isomorphism between $(X,x)$ and the special fiber $\pi^{-1}(0)$. 
  A \emph{1-parameter smoothing} is a deformation over a germ of smooth
  curve such that the generic fibers are smooth.
  
 A deformation $\pi:(Y,y)\rightarrow (S,0)$ of $(X,x)$ is {\it
    versal} if any other deformation is 
  obtainable from it by a base-change. 
 A versal deformation is {\it  miniversal} if  
  the Zariski tangent space of its base $(S,0)$ has the smallest possible 
  dimension. 
 A {\it smoothing component} is an irreducible component 
  of the miniversal 
  base space  over which the generic  fibers are smooth. 
\end{definition} 
 
If $(X,x)$ is a germ of reduced complex analytic space \emph{with 
an isolated singularity}, then the following well-known facts 
hold: 
 
  (i)  (Schlessinger \cite{S 68}, Grauert \cite{G 72}) The miniversal 
  deformation $\pi$ 
  exists and is unique up   to (non-unique) isomorphism. 
 
(ii) There exist (Milnor) representatives $Y_{red}$ and $S_{red}$ 
of the reduced total and base spaces of $\pi$ such that the 
restriction $\pi:\partial Y_{red} \cap\pi^{-1}(S_{red}) 
\rightarrow S_{red}$ is a trivial $C^{\infty}$-fibration. 
 
\vspace{2mm} 
 
Hence, for each smoothing component of $(X,x)$,  the oriented
diffeomorphism type  of the oriented manifold with boundary  
 $(\pi^{-1}(s)\cap Y_{red},\pi^{-1}(s)\cap \partial Y_{red})$ does not depend 
 on the choice of the generic element   $s$: it
 is called \emph{the Milnor fiber} of that component. Moreover, 
its boundary   is canonically 
identified with the boundary $\partial(X,x)$ by an
orientation-preserving diffeomorphism,  \emph{up to isotopy}. 

Therefore, if one takes the Milnor fibers corresponding to two
distinct components of $S_{red}$, one may ask whether they are
\emph{strongly diffeomorphic} (a notion introduced in subsection
\ref{conv})  with respect to the previous natural
identification of their boundaries.

\subsection{The $\delta$-invariant and $\delta$-constant
  deformations.} \label{delta}
 
 Let $(C,0)$ be a (not necessarily plane) curve singularity. Its
 $\delta$-\emph{invariant} is by definition:
  $$ \delta(C,0):= \dim_{\C}( \nu_* \mathcal{O}_{\tilde{C},
    \tilde{0}} / \mathcal{O}_{C,0}),$$
where $(\tilde{C}, \tilde{0}) \stackrel{\nu}{\rightarrow} (C,0)$
is a normalization morphism (in particular, $(\tilde{C}, \tilde{0})$
denotes a multi-germ). 
If $C$ is a reduced curve, by definition, its $\delta$-\emph{invariant} is the
sum of the $\delta$-invariants of all its singular points. 

As an example, consider an \emph{ordinary $m$-tuple
  point} of a reduced curve embedded in a smooth surface, that
is a point where the curve has exactly $m$ local irreducible components, all 
smooth and meeting pairwise transversally. At such a point, the
$\delta$-invariant of the curve is equal to $\frac{m(m-1)}{2}$. 

More generally, if $(C,0)$ is a germ of plane curve, one has the
following classical formula:
\begin{equation} \label{deltgen}
 \delta(C,0)= \sum_{P} \frac{m(C,P)(m(C,P)-1)}{2}
\end{equation} 
where $P$ varies among all the points infinitely near $0$ on the
ambient surface and $m(C,P)$ denotes the multiplicity of the strict
transform of $C$ at such a point $P$ (see \cite[pages 298 and 393]{H 77}).

Denote by $(\Sigma,0)  \stackrel{\pi}{\rightarrow} (S,0)$ a
deformation of $(C,0)$ over a smooth germ of curve. It is called
$\delta$-\emph{constant} if there exists a Milnor representative of
$\pi$ such that the $\delta$-invariant of its fibers is constant. 
By a characterization of Teissier  \cite{T 80} (see also 
\cite{CL 06}), a deformation $\pi$
is $\delta$-constant if and only if the normalization
morphism of $\Sigma$ has the form $\tilde{C}\times S\to \Sigma$, and it is
the \emph{simultaneous normalization} of the fibers of $\pi$. 
In particular, if one has a $\delta$-constant embedded deformation of a germ of
\emph{plane} curve $(C,0)$ such that the general fiber $C_s$ has only
ordinary multiple points, then after a choice of a Milnor
representative of the morphism, the general fibers are immersed
\emph{discs} in a $4$-ball.

\subsection{The topology of normal surface singularities}
  \label{toponorm}

Consider a normal surface singularity $(X,x)$. It has a unique 
minimal normal crossings resolution,
that is, a resolution whose exceptional set $E$ is a divisor with normal
crossings and which looses this property if one contracts any
$(-1)$-curve of it. Consider also the \emph{weighted dual graph} of this
resolution. Each vertex is weighted by the genus and by the
self-intersection number of the corresponding irreducible component of
the divisor $E$. 

From the resolution, the abstract boundary
$\partial(X,x)$ inherits a \emph{plumbing structure}, that is, a
family of pairwise disjoint embedded tori whose complement is fibered by
circles and such that on each torus, the intersection number of the
fibers from each side is 
$\pm 1$. The tori correspond to the edges of the dual graph and the
connected components of their complement correspond to the
vertices.  For more details, see \cite{N 81} or \cite{PP 07}.

Neumann proved in \cite{N 81} that the knowledge of the boundary of
$(X,x)$, up
to orientation-preserving homeomorphisms, is equivalent to the
knowledge of the weighted dual graph of the minimal normal crossings
resolution, up to isomorphisms. In \cite[Theorems 9.1 and 9.7]{PP 07}
was proved the following enhancement of Neumann's result:

\begin{theorem} \label{invar}
  The plumbing structure corresponding to the minimal normal crossings
  resolution is determined by the oriented boundary (as abstract manifold);
  in paticular, it is invariant, up to isotopy, by all orientation-preserving 
  diffeomorphisms of the boundary. 
\end{theorem}

It is important to note that in the definition of a plumbing
structure, the circles of the fibration of the complement of the tori
are considered \emph{unoriented} (otherwise the previous theorem would
not be true). But in the case of the boundary of a surface
singularity, once a fiber is oriented, there is a canonical way 
to orient all of them (see \cite[Corollary 8.6]{PP 07}). 
 
 Note also that any fiber has a natural framing, given by the nearby
 fibers. We say that it is \emph{the framing coming from the plumbing
   structure}.

\medskip 
\section{The  smoothings of sandwiched surface 
singularities, \\ 
following  de Jong and van Straten} 
\label{topsmooth} 

de Jong and van Straten 
related in \cite{JS 98}  the deformation theory of sandwiched 
surface singularities to the deformation theory of so-called 
\emph{decorated plane curve 
  singularities}. They showed that 1-parameter deformations of
particular 
  decorated curves provide 1-parameter deformations for sandwiched 
  singularities, and that all of these later ones can be obtained in 
  this way. Moreover, the Milnor fibers of those  which are smoothings 
  can be combinatorially described by  the so-called \emph{picture 
    deformations}. The aim of this section is to recall briefly this
  theory. 
\medskip

\subsection{Basic facts about sandwiched singularities}\label{SAND}

 A normal surface 
singularity $(X,x)$ is called {\it sandwiched} (see \cite{S 90}) if it
is analytically isomorphic to a germ 
of algebraic surface which admits a birational map $X\to \C^2$.  See 
also \cite{JS 98, L 00} for other view-points. 
 
 Sandwiched singularities are rational, therefore their minimal normal
 crossings resolution coincides with their minimal resolution and the
 associated weighted graph is a negative-definite tree of rational
 curves. They are  
characterized (like the rational singularities) by their dual 
resolution graphs, which allows to speak about \emph{sandwiched 
graphs}: 
 
\begin{proposition} \label{charsand} 
A graph $\Gamma$ is sandwiched  if  by gluing (via new edges) to some of the
vertices of $\Gamma$  new  rational vertices with self-intersections $-1$, 
one may obtain a `smooth graph', i.e.  the dual tree of  
a configuration of ${\mathbb P}^1$'s which blows down to a smooth 
point. 
\end{proposition} 

The way to add such $(-1)$-vertices is not unique. But once
such a choice was done, for each sandwiched singularity with that
topology, one may embed a tubular neighborhood of the exceptional set
$E$ of the minimal resolution in a larger surface which contains also some
$(-1)$-curves, whose union with $E$ can be contracted and has as dual graph  
the chosen `smooth graph'. This procedure will be used in the next
subsection.

\subsection{Decorated curves and their deformations}\label{decor} 
Any sandwiched singularity may be obtained from a weighted curve 
$(C,l)$. Here $(C,0)\subset (\C^2,0)$ denotes a reduced germ of 
plane curve \emph{with numbered branches}
(irreducible components) $\{C_i\}_{1 \leq i \leq r}$
and $(l_i)_{1 \leq i \leq r} \in (\N^*)^r$. Let us explain this.

Consider the 
minimal abstract resolution of $C$ obtained by a sequence of 
blowing-ups of (closed) points. If one realizes those blowing-ups
inside the smooth 
ambient surface, one does not necessarily obtain an embedded
resolution of $C$, that is, the total transform of $C$ is not
necessarily a divisor with normal crossings.  

The \emph{multiplicity  sequence}   
associated with $C_i$ is the sequence of multiplicities on the 
successive strict transforms of $C_i$, starting from $C_i$ itself 
and not counting the last strict transform.  The {\it total 
multiplicity} $m(i)$  of $C_i$ with respect to $C$ is the sum of 
the sequence of multiplicities of $C_i$ defined before.

\begin{definition} \cite[(1.3)]{JS 98} \label{decgerm} 
  A {\it decorated germ} of plane curve is a weighted germ $(C,l)$ 
  such that $l_i\geq m(i)$ for all \ $  i \in \{1,\ldots,r\}$. 
\end{definition} 
 
The point is that,   
starting from a decorated  germ, one can blow up iteratively points 
infinitely near $0$ on the strict transforms of $C$, such that the sum
of multiplicities of the strict transform of $C_i$ at 
such points is exactly $l_i$. Such a composition of blow-ups is
determined canonically by $(C,l)$. 
If $l_i$ is sufficiently  large (in general this bound is larger than $m(i)$,
see below),  then the union of the 
exceptional components which do not meet the strict transform of 
$C$ form a \emph{connected} configuration of (compact) curves $E(C,l)$. After
the contraction of $E(C,l)$, one gets a sandwiched singularity 
$X(C,l)$, determined uniquely by $(C,l)$ (for details see  
\cite{JS 98}). 
 
Using Proposition \ref{charsand}, for any sandwiched singularity $(X,x)$
one can find  $(C,l)$ such that $(X,x)$ can be represented as $X(C,l)$. 
Indeed, choose an
extension of the sandwiched graph to a smooth graph by adding
$(-1)$-vertices (different choices provide different realizations
$X(C,l)$). 
Fix a geometrical realization of this extension starting from
the minimal resolution of the given sandwiched
singularity (by gluing tubular neighbourhoods of the new 
$(-1)$-curves). Consider then a \emph{curvetta} (that is, a smooth transversal
slice through a general point) for each $(-1)$-component and blow down the 
union of these curvettas to the smooth surface obtained by contracting
the configuration with smooth graph. 
Number then the components of the plane germ $C$
obtained in this way and consider the function $l$ describing the
reverse process of reconstruction of the initial configuration of
compact curves and curvettas. One obtains a decorated germ $(C,l)$
such that $X(C,l)$ is the original singularity.

We introduce the
following definition:  

\begin{definition}
$(C,l)$ is called  \emph{a standard decorated germ} if it was obtained in the
previous way. 
\end{definition}

An example of the previous process applied to a cyclic quotient
singularity is indicated in Figure \ref{fig:Extgraf}. A representative
of the standard decorated germ obtained like this is $C=C_1 \cup C_2$,
where $C_1$ is defined by $y^2-x^3=0$, $C_2$ is defined by $y=0$ and
$l_1=6, l_2=3$. One has $m(1)=3, m(2)=2$. 

\medskip 
\begin{figure}[h!] 
\labellist \small\hair 2pt 
\pinlabel{$-3$} at 3 112
\pinlabel{$-2$} at 58 112
\pinlabel{$-3$} at 111 112
\pinlabel{$-1$} at 238 112
\pinlabel{$-3$} at 291 112
\pinlabel{$-2$} at 347 112
\pinlabel{$-3$} at 401 112
\pinlabel{$-1$} at 323 37
\pinlabel{$C_1$} at 637 141
\pinlabel{$C_2$} at 510 72
\pinlabel{$6$} at 662 28
\pinlabel{$3$} at 662 82
\endlabellist 
\centering 
\includegraphics[scale=0.60]{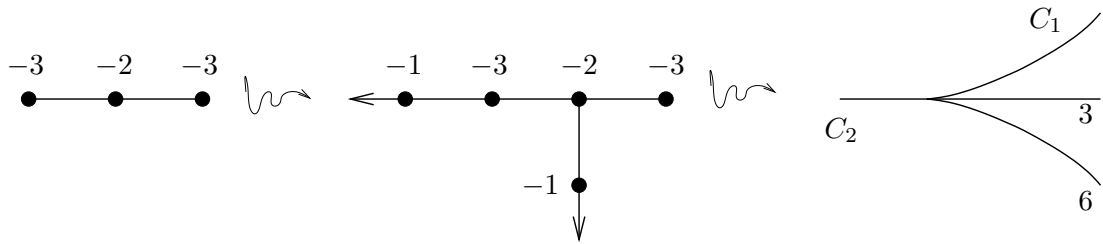} 
\vspace{2mm} \caption{A construction of standard decorated germ} 
\label{fig:Extgraf} 
\end{figure}

We wish to emphasize again, that for standard decorated germs the integers
$l_i$ usually are larger than $m(i)$. The `right'  bound is
the following.  
Consider the \emph{minimal embedded resolution} of $(C,0)
\hookrightarrow (\C^2,0)$ by blowing-ups of closed points. That is,
one does not only ask that the strict transform of $C$ 
be smooth, but also that its total transform be a normal
crossings divisor. For each $i \in \{1,...,r\}$, denote by $M(i)$ (see
\cite[page 456]{JS 98}) the sum of multiplicities of the strict
transforms of $C_i$, before arriving at this embedded resolution (for
instance, for the germ represented in Figure \ref{fig:Extgraf}, one
gets $M(1)=5$ and $M(2)=2$).  One may see easily that $(C,l)$ is a
standard decorated germ 
if and only if $l_i \geq M(i) +1$, for all $i \in \{1,...,r\}$. 

Consider again an \emph{arbitrary} decorated germ $(C,l)$. 
The total multiplicity of $C_i$ with respect to $C$ may be encoded 
also as the unique subscheme of length $m(i)$ supported on the 
preimage of $0$ on the normalization of $C_i$. This allows to define
the \emph{total  
multiplicity scheme} $m(C)$ of any reduced curve contained in a 
smooth complex surface, as the union of the total multiplicity 
schemes of all its germs. The deformations of $(C,l)$ considered by de
Jong and van Straten are: 
 
\begin{definition} (i) \cite[(4.1)]{JS 98} \label{defcurve} 
  Given a smooth complex analytic surface $\Sigma$, a  pair $(C,l)$ consisting 
  of a reduced curve  $C\hookrightarrow \Sigma$ and a subscheme $l$ of 
  the normalization $\tilde{C}$ of $C$ is called {\it a decorated 
    curve} if $m(C)$ is a subscheme of \, $l$. 
 
\vspace{1mm} 
 
 (ii) \cite[p. 476]{JS 98} 
  A $1$-{\it parameter deformation} of a decorated curve $(C,l)$ 
  over a germ of smooth curve $(S,0)$ consists of: 
 
(1) a $\delta$-constant deformation $C_S \rightarrow S$ of $C$; 
 
(2) a flat deformation $l_S \subset \tilde{C_S} =\tilde{C}\times 
S$ of the scheme $l$, such that: 
 
(3) $m_S\subset l_S$, where the \emph{relative total multiplicity 
  scheme} $m_S$ of $\tilde{C_S}\rightarrow C_S$ is defined as the 
closure $\overline{\bigcup_{s \in S \setminus 0} m(C_s)}$. 
 
\vspace{1mm} 
 (iii)  A $1$-parameter deformation $(C_S, l_S)$ is called a 
   {\it picture deformation} if for generic $s \neq 0$ the divisor 
   $l_s$ is reduced. 
\end{definition} 

It is an immediate consequence of the definitions that for a picture
deformation $C_S$, the singularities of $C_{s \neq 0}$ are only
ordinary multiple points. Therefore, it is easy to draw a real picture
of such a deformed curve, which motivates the name. As an example, in
Figure \ref{fig:Pictdef} we represented a generic fiber $C_{s\neq 0}$ and the
image of the support of the subscheme $l_s$ of its normalization for
one of the picture deformations of the standard decorated curve obtained in the
Figure \ref{fig:Extgraf}.

\medskip 
\begin{figure}[h!] 
\labellist \small\hair 2pt 
\pinlabel{$(C_1)_s$} at 332 135
\pinlabel{$(C_2)_s$} at 12 52
\endlabellist 
\centering 
\includegraphics[scale=0.55]{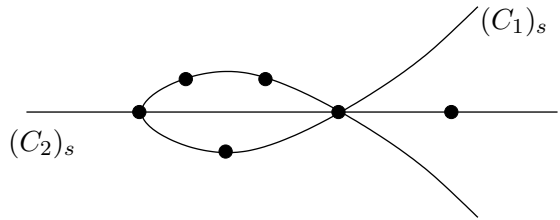} 
\vspace{2mm} \caption{A picture deformation of a standard decorated germ} 
\label{fig:Pictdef} 
\end{figure}

de Jong and van Straten proved that part of the deformation theory of
sandwiched surface singularities may 
be reduced to that of $1$-parameter deformations of decorated germs:
 
\begin{theorem} \cite[(4.4)]{JS 98} All the 1-parameter
  deformations of  
$X(C,l)$ are obtained by $1$-parameter deformations of the 
decorated germ $(C,l)$. Moreover, \emph{picture 
  deformations} provide all the  smoothings of $X(C,l)$. 
\end{theorem}

\subsection{Picture deformations and the associated Milnor
  fibers.}\label{pict}  
Consider a decorated germ $(C,l)$ and one of its picture 
deformations $(C_S, l_S)$. Fix a closed Milnor ball $B$ for the 
germ $(C,0)$. For $s\neq 0$ sufficiently small, $C_s$ will have a 
representative in $B$, denoted by $D$, which meets $\partial B$ 
transversally. As explained in  subsection \ref{delta}, 
it is a union of \emph{immersed discs} 
$\{D_i\}_{1\leq i\leq r}$ canonically oriented by their complex 
structures (and whose set of indices correspond canonically to 
those of  $\{C_i\}_{1\leq i\leq r} $). The singularities of $D$ 
consist of ordinary $m$-tuple points, for various $m\in \N^*$. 
 
Denote by $\{P_j\}_{1 \leq j \leq n}$ the set of images in $B$ of the 
points in the support of $l_s$. It is a finite set of points which 
contains the singular set of $D$ (because $m_s \subset l_s$ for $s
\neq 0$), but 
it might contain some other `free' points as well. There is a priori no 
preferred choice of their ordering. Hence, the matrix introduced 
next is well-defined only \emph{up to permutations of columns}. By
contrast, the ordering
of lines is determined by the fixed order of the components of $C$. 
 
\begin{definition} \cite[page 483]{JS 98} \label{incid} 
  The {\it incidence matrix} of a picture deformation $(C_S, l_S)$ 
  is the matrix $\mathcal{I}(C_S, l_S)\in \mbox{Mat}_{r,n}(\Z)$ 
(with $r$ rows and $n$ columns) whose 
  entry $m_{ij}\in \N$ at the  intersection 
  of the $i$-th row and the $j$-th column is equal to the multiplicity
  of $P_j$ as a point of $D_i$. 
\end{definition} 

Such a matrix satisfies the following necessary conditions:
\begin{equation} \label{delt} 
  \sum_{j=1}^{n}\frac{m_{ij}(m_{ij}-1)}{2} =  \delta_i, 
    \mbox{ for all } i \in \{1,...,r\}, 
\end{equation} 
where $\delta_i$ denotes the $\delta$-invariant of the branch $C_i$;
\begin{equation} \label{inters} 
   \sum_{j=1}^{n} m_{ij} m_{kj} = C_i \cdot C_k,  \mbox{ for all } i,
   k    \in \{1,...,r\}, i \neq k, 
\end{equation} 
where $C_i\cdot C_k$ stands for the intersection multiplicity at $0$
of the branches $C_i$ and $C_k$;
\begin{equation} \label{nopoints} 
  \sum_{j=1}^{n}m_{ij} =  l_i, \mbox{ for all } i \in \{1,...,r\}. 
\end{equation} 

The previous equations are consequences of Definition 
\ref{defcurve} (ii) and of the fact that $C_{s\neq 0}$ has only
ordinary singularities.

It is an open problem to determine, in general, the set of incidence matrices
associated with a given decorated germ, among the matrices which satisfy
simultaneously (\ref{delt}), (\ref{inters}) and (\ref{nopoints}). It
is also an open problem to 
determine the union of the sets of incidence matrices associated with
all decorated germs with a given topology. 

The Milnor fiber of a smoothing associated to a picture deformation of
a standard decorated germ $(C,l)$ is recovered as follows. Let: 
\begin{equation} \label{blowpoints} 
   (\tilde{B}, \tilde{D}) 
   \stackrel{\beta}{\longrightarrow}(B, D) 
\end{equation} 
  be the simultaneous blow-up of  the points $P_j$ of $D$. Here 
  $\tilde{D}:= \cup_{1\leq i \leq r}\tilde{D}_i$, where $\tilde{D_i}$ 
  is the strict transform of the disc $D_i$ by the modification
  $\beta$.  
Consider pairwise disjoint compact tubular neighborhoods of the discs
$\tilde{D}_i$  in $\tilde{B}$ and denote their interiors by $T_i$. 
 
\begin{proposition}\cite[(5.1)]{JS 98} \label{recofiber} 
  Let $(C,l)$ be a standard decorated germ. Then 
  the Milnor fiber of the smoothing of $X(C,l)$ corresponding to the 
  picture deformation $(C_S, l_S)$ is orientation-preserving
  diffeomorphic to the  
  compact oriented manifold with boundary $W:= \tilde{B}\setminus 
         (\bigcup_{1\leq i \leq r} T_i)$ (whose corners are smoothed). 
\end{proposition}

Moreover, one can show that in this realization, the canonical identification
of the boundaries of the Milnor fibers corresponding to different
smoothing components is  the unique identification  (up to isotopy) which
extends the identity morphism on the complement in $\partial \tilde{B}$ of some
tubular neighborhood of $\partial C$. Even more, by this presentation (for a
fixed  $(C,l)$) one can also see the canonical identification of all the
boundaries of the Milnor fibers with the boundary $\partial(X,x)$, 
cf. also Lemma \ref{handles} and the paragraphs following it.

We mention that although many results regarding decorated curves are valid for
any $l_i\geq m(i)$, in Proposition \ref{recofiber} one needs  $l_i\geq
M(i)+1$, i.e.  
standard decorated curves (a fact omitted  in \cite[(5.1)]{JS 98}).
Without this hypothesis, Proposition \ref{recofiber} 
is false, as one may see by considering more than one
curvetta attached to a $(-1)$-curve.

\medskip 
\section{From a Milnor fiber to the incidence matrix} 
\label{ljs}

In this section we prove our main results, Theorem \ref{distinc} and Theorem
\ref{lowbound}. The main idea of the proofs, inspired by the article 
\cite{L 08} of Lisca, is to close all the Milnor fibers with the same
`cap' (a 4-manifold with boundary) which is  glued each time
in the `same way' (this make sense thanks to the canonical
identifications up to isotopy of the 
boundaries of the Milnor fibers). 
The cap we use emerges naturally from the view-point
of de Jong and van Straten. 

All over the section, we suppose that $(C,l)$ is a \emph{standard} decorated
germ. We keep the notations of section \ref{topsmooth}.

\subsection{Markings of oriented sandwiched 3-manifolds.} 
   \label{mark}

In this subsection we reinterpret the ambiguity of the choice of the embedding
of the fixed  sandwiched graph into a smooth one as a \emph{marking}.

\begin{definition} \label{sanfold}
   An oriented 3-manifold is called a \emph{sandwiched
     manifold} if it is orien\-tation-preserving diffeomorphic to the
   oriented boundary $\partial(X,x)$ of a sandwiched surface
   singularity $(X,x)$. Any such singularity is called a
   \emph{defining singularity} of $M$. 
\end{definition}

By Neumann's theorem recalled in subsection \ref{toponorm}, sandwiched
3-manifolds determine the sandwiched 
graph associated with any defining singularity. By Theorem \ref{invar},
they may also be endowed canonically, up to an isotopy, with a plumbing
structure $\mathcal{T}_{min}$ corresponding to the minimal resolution
of a defining singularity. We say that the connected components of the
complement in $M$ of the tori of $\mathcal{T}_{min}$ are the
\emph{pieces} of $M$.

Suppose now that $(X,x)=X(C,l)$, where $(C,l)$ is a standard
decorated germ with $(C,0)\subset (\C^2,0)$. 
Let $(\Sigma, E(\pi)) \stackrel{\pi}{\rightarrow}
(\C^2,0)$ be the composition of blow-ups of points infinitely near $0$
determined by $(C,l)$. One may write:
  $$E(\pi)= E(C,l) + \sum_{i=1}^r E_i,$$
where $E(C,l)$ is the exceptional divisor of the minimal resolution of 
$(X,x)$ and $(E_i)_{1 \leq i \leq r}$ are all the $(-1)$-curves contained in
the exceptional divisor $E(\pi)$, numbered such that $E_i$ is the
unique irreducible component of $E(\pi)$ which intersects the strict
transform of $C_i$. 
Denote by $F_i$ the unique irreducible component
of  $E(C,l)$ which intersects $E_i$. To $F_i$ corresponds a
well-defined piece of $M$. In this way, one gets a map from the set
$\{1,...,r\}$ to the set of pieces of $M$.

\begin{definition}
   A map from $\{1,...,r\}$ to the set of pieces of $M$ obtained as
   above is called \emph{a marking} of $M$. 
\end{definition}

We see that each choice of defining singularity of $M$ of the form
$X(C,l)$, where $(C,l)$ is a standard decorated germ, determines a
well-defined marking of $M$.

\subsection{Closing the boundary of the Milnor fiber.}\label{closemiln} 

  Let $(C_S, l_S)$ be a picture
deformation of the decorated germ $(C,l)$. 
 
As the disc-configuration $D$ is obtained by deforming $C$, 
 its  boundary $\partial D:= \linebreak \cup_{1\leq i \leq r}D_i 
  \hookrightarrow \partial B$ is isotopic as an 
  oriented link to $\partial C\hookrightarrow \partial B$. Therefore, 
  we can isotope $D$ 
  outside a compact ball containing all the points $P_j$ 
  till its boundary  coincides with the boundary of $C$. From now on,
  $D$ will denote the result of this isotopy. 
  Let $(B', C')$ be a second copy of $(B,C)$, and define (see Figure
  \ref{fig:Glue}): 
  $$(V,\Sigma):=(B,D)\cup_{id}(\overline{B}', \overline{C}').$$

\begin{figure}[h!] 
\labellist \small\hair 2pt 
\pinlabel{$B$} at 52 232
\pinlabel{$\overline{B}'$} at 230 232
\pinlabel{$D$} at 60 84
\pinlabel{$\overline{C}'$} at 252 84
\pinlabel{$\partial B= \partial \overline{B}'$} at 285 290
\endlabellist 
\centering 
\includegraphics[scale=0.55]{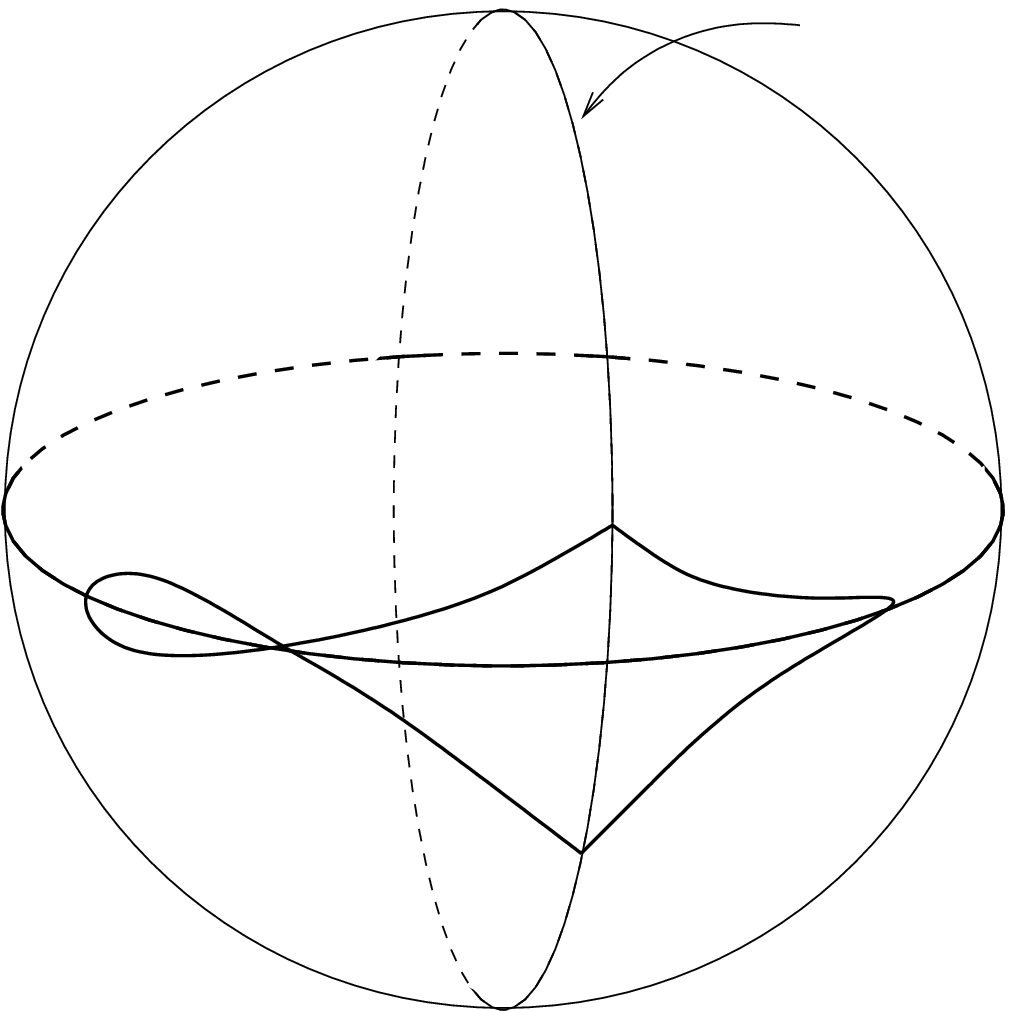} 
\vspace{2mm} \caption{Gluing a curve and one of its picture deformations} 
\label{fig:Glue} 
\end{figure}

Here $V$ is 
  the oriented 4-sphere obtained by gluing the boundaries 
  of $B$ and $\overline{B}'$ by the tautological identification, and 
  $\Sigma:= \cup_{i=1}^r\Sigma_i$, where 
   $\Sigma_i$ is obtained by gluing 
  $D_i$ (perturbed by the above isotopy) and $\overline{C}_i'$ 
  along their common boundaries. Recall that   $\overline{C}_i'$ 
is a topological disc, while $D_i$ an immersed disc. 
Moreover,  one can also glue  $(\overline{B}', \overline{C}')$ 
with $(\tilde{B},\tilde{D})$ in such a way that the blow-up morphism 
$\beta$ of  (\ref{blowpoints}) may be extended by the identity on 
$\overline{B}'$, yielding:  
$$ (\tilde{V}, \tilde{\Sigma})\stackrel{\beta}{\longrightarrow} (V, 
\Sigma).$$ Here $\tilde{\Sigma}:= \cup_{i=1}^r \tilde{\Sigma}_i$, 
where $\tilde{\Sigma}_i$ denotes the strict transform of
$\Sigma_i$, i.e. $\tilde{\Sigma}_i = \tilde{D}_i \cup 
\overline{C}'_i$. In particular, 
 $\tilde{\Sigma}_i$ is a topologically {\em embedded} 2-sphere in
$\tilde{V}$, which is smoothly embedded in $\tilde{V} \setminus 0'$ (it
is smoothly embedded everywhere if and only if the branch $C_i$ is
smooth). 

 \begin{lemma} \label{intsp} 
     The intersection numbers of the oriented 
     spheres $(\tilde{\Sigma}_i)_{1\leq i \leq r}$ inside the oriented 
     $4$-manifold $\tilde{V}$ are the following: 
     $$\left\{ \begin{array}{ll} 
             \tilde{\Sigma}_i^2=-l_i -2 \delta_i & \mbox{for all} \ 
                    i \in \{1,\ldots,r\}, \\ 
             \tilde{\Sigma}_i \cdot \tilde{\Sigma}_j = - C_i \cdot C_j & 
             \mbox{for all}   \  i <j. 
           \end{array} \right.$$ 
   \end{lemma} 
 
\begin{proof} Fix $i \in \{1,...,r\}$ and, as in Definition
  \ref{incid},  denote by $m_{ij}$ the
  multiplicity of $P_j$ as a point of the curve $D_i$. By combining
  the equations (\ref{delt}) and (\ref{nopoints}), we get: 

  \begin{equation} \label{sumsquare}
     \sum_{j=1}^n m_{ij}^2= l_i + 2 \delta_i.
  \end{equation}

  By blowing-up a point of multiplicity $m \in \N^*$ of a compact
  complex curve on a smooth complex surface, its self-intersection
  drops by $m^2$. This is also true if one looks simply at a $2$-cycle
  on an oriented 4-manifold which becomes holomorphic in
  an adequate chart containing the point to be blown-up. Applying this
  last fact to $\Sigma_i \hookrightarrow V$, which is holomorphic as
  seen inside a neighborhood of the points $P_j\in B\hookrightarrow
  \C^2$, we get: 
   $$\tilde{\Sigma}_i^2= \Sigma_i^2 -\sum_{j=1}^n m_{ij}^2
   \stackrel{(\ref{sumsquare})}{=} -l_i -2 \delta_i,$$
since $\Sigma_i^2=0$ by the fact that $H_2(V)=0$.
  The first family of formulae is proved. 

  The second family of formulae results simply from the fact that
  $\tilde{\Sigma}_i$ and $\tilde{\Sigma}_j$ intersect only at $0'$, and
  that $\overline{B}'$ is a copy of $B$ with changed orientation.  
  \end{proof}

Write $T:=\bigcup_{1\leq i \leq r} T_i$ and set 
also:
\begin{equation} \label{defU} 
  U:= \overline{B}'\cup T. 
\end{equation}

Since $W=\tilde{B}\setminus T$ (cf. \ref{recofiber}), the closed
oriented 4-manifold $\tilde{V}$ 
is obtained by closing the boundary of $W$ by the cap $U$. 
This cap is independent of the chosen picture deformation:

\begin{lemma} \label{handles} $U$ depends only on $(C,l)$ and
is independent of the chosen picture 
deformation (therefore one may close all the different Milnor 
fibers using the same $U$). In fact, each $T_i$ in $U$ is a 
$4$-dimensional handle of index $2$ glued to 
    $\overline{B}'$ along the knot $\partial C_i\hookrightarrow 
    \partial \overline{B}'$ endowed with the $(-l_i-2 \delta_i)$-framing. 
  \end{lemma}

\begin{proof} The chosen framing index is equal to the self-intersection
  number of a surface obtained by gluing the core of the handle to a
  2-chain contained in the 4-ball $\overline{B}'$ and whose boundary
  is $\partial C_i$. By taking $\tilde{D}_i$ as core disc and
  $\overline{C}_i'$ as 2-chain, we get the sphere
  $\tilde{\Sigma}_i$. We apply then the first family of formulae
  stated in Lemma \ref{intsp}. 
  \end{proof}

 Moreover, the cap $U$ is always glued in the same way, up to isotopy,
 to the boundaries of the Milnor fibers corresponding to the various picture
 deformations, identified by the natural diffeomorphisms. To see this,
 note first that $T_i$ may be seen as a  2-handle attached to
 $\overline{B}'$, with core disc $\tilde{D}_i$. Take a co-core
 $K_i$. It is a core of $T_i$ seen as a 2-handle attached to $W$. In
 order to describe the attaching of $U$ to $W$, it is enough to
 describe the link $\cup_{i=1}^r \partial K_i$ in $W$ and its
 associated framing up to isotopy. 
 
In the following lemma, we also use the notations of subsection
\ref{mark}. 

\begin{lemma} \label{attach}
  Let $(C,l)$ be a  standard decorated germ. 
  Consider on $\partial W$ the plumbing structure corresponding to the
  minimal resolution of $X(C,l)$. Consider then an ordered collection
  of $r$ regular fibers $f_1,...,f_r$, where $f_i$ is
the generic fiber of the piece of $W$ marked by $i$. 
Each knot $f_i$ is endowed with the framing coming from
  the plumbing structure, increased by $1$.
 Then one obtains a framed
  link isotopic to   the attaching framing of $T$. 
\end{lemma}

\begin{proof} Note that the orientation of $\partial W$ (as induced
  from $W$) coincides with the orientation induced from
  $\overline{U}$. In what follows, we will reason inside 
    $\overline{U}= B' \cup \overline{T}$. Consider on  $B'$
   the natural complex structure, as a copy of
   $B$. Then $C' \hookrightarrow \tilde{\overline{\Sigma}}_i$ is
   holomorphic with respect to this complex structure and we may
   consider the sequence of blow-ups above $0' \in B'$ dictated by the
   copy $(C', l')$ of $(C,l)$. Let $\overline{U}_{\alpha}
   \stackrel{\alpha}{\rightarrow} \overline{U}$ be this total blow-up
   morphism and $\Sigma_i^{\alpha}$ be   the strict transform of
   $\tilde{\overline{\Sigma}}_i$ by the morphism $\alpha$.  

  Denote by $(m_0(i),...,m_{I(i)}(i))$ the multiplicity sequence
  associated to $C_i \hookrightarrow C$. The total multiplicity of
  $C_i$ with respect to $C$ is given by $m(i) = \sum_{k=0}^{I(i)}
  m_k(i)$. One has $l_i = m(i) +h_i$, where $h_i \in \N$, which
  implies that:
  \begin{equation} \label{comp}
    (\Sigma^{\alpha})^2 = (\tilde{\overline{\Sigma}}_i)^2 -
    \sum_{k=0}^{I(i)} m_k(i)^2 - h_i.
  \end{equation}
  But $\delta_i = \sum_{k=0}^{I(i)} \frac{m_k(i)(m_k(i)-1)}{2}$ by
  formula (\ref{deltgen}), which
  implies that $\sum_{k=0}^{I(i)} m_k(i)^2 = 2 \delta_i
  +m(i)$. Therefore:
  $$ \sum_{k=0}^{I(i)} m_k(i)^2 + h_i = 2\delta_i +m(i) +h_i = 2
  \delta_i +l_i.$$
  From formula (\ref{comp}) and Lemma \ref{intsp}, we deduce that
  $(\Sigma^{\alpha})^2 =0$. 

  The total space $\overline{U}_{\alpha}$ appears then as a tubular
  neighborhood of a configuration of spheres whose dual graph is
  obtained from the dual graph of the exceptional divisor of $\alpha$
  by attaching a $(0)$-vertex to each $(-1)$-vertex.  The co-cores
  of the handles $\overline{T}_i$ appear as curvettas for
  $\Sigma_i^{\alpha}$. An easy application of plumbing calculus (see
  \cite{N 81}) finishes the proof. 
\end{proof}

 \subsection{The reconstruction of the incidence matrix from the
   Milnor  fiber} \label{mainthm}
 
Fix  a picture deformation  $(C_S, l_S)$ of $(C,l)$. 
 
 \begin{lemma}
   Up to permutations, there exists a unique basis
   $(e_j)_{1 \leq j \leq n}$ formed by   $(-1)$-classes  
   of $H_2(\tilde{V})$ such that the matrix:
     $$N(C_S, l_S):= ([\tilde{\Sigma}_i]\cdot e_j)_{1\leq i \leq r,
       1\leq j \leq n}$$ has only non-negative entries. 
 \end{lemma}

 \begin{proof} Denote by $\sigma_j:= \beta^{-1}(P_j)$ the 
   2-spheres embedded in $\tilde{V}$ obtained by blowing-up the points
   $P_j$. We suppose them oriented by the complex structure existing
   on $\tilde{B}$. Their homology classes are $(-1)$-classes in 
   $H_2(\tilde{V})$. As
   $V$ is diffeomorphic to a 4-sphere, we deduce that $([\sigma_j])_{1 \leq
     j \leq n}$ is a basis of the free group
   $H_2(\tilde{V})$. Moreover, the spheres being disjoint, the
   intersection form on $H_2(V)$ is in canonical form (the opposite of
   a sum of squares) in this basis. This shows that any $(-1)$-class
   is equal to some $\pm[\sigma_j]$. 

   By construction, if we take the basis $([\sigma_j])_{1 \leq
     j \leq n}$, the matrix $([\tilde{\Sigma}_i]\cdot
   [\sigma_j])_{1\leq i \leq r, 
       1\leq j \leq n}$ has only non-negative entries (in fact,
     elements of $\{0, +1\}$). Suppose that we
     change the sign of one of the classes $[\sigma_j]$. As the corresponding
     column contains at least one non-zero entry (because the point
     $P_j$ lies on at least one of the curves $D_i$), we deduce that
     the new matrix has some negative entries in that column. This
     proves the lemma.  
 \end{proof}

 \begin{proposition} \label{eqmat}
Let  $\mathcal{I}(C_S, l_S)$ be the  incidence matrix of the picture
deformation $(C_S,l_S)$. Then 
   the matrices $\mathcal{I}(C_S, l_S)$ and $N(C_S, l_S)$ are equal, up to
    permutations of columns. 
 \end{proposition}

\begin{proof} The previous proof shows that, up to permutation of
  columns:
  $$N(C_S, l_S)= ([\tilde{\Sigma}_i]\cdot [\sigma_j])_{1\leq i \leq r,
       1\leq j \leq n}.$$
  But $[\tilde{\Sigma}_i]\cdot [\sigma_j]$ is equal to the incidence number
  between the curve $D_i$ and the point $P_j$, that is, to the
  $(i,j)$-entry $m_{ij}$ of the matrix $\mathcal{I}(C_S, l_S)$. The
  proposition is   proved. 
\end{proof}
 
The next theorem explains that one may reconstruct the incidence
matrix of a picture deformation of $(C,l)$ from the associated Milnor
fiber with marked boundary. 

\begin{theorem} \label{markinc}
  The incidence matrix $\mathcal{I}(C_S, l_S)$ associated to a picture
  deformation of a standard decorated germ $(C,l)$ is determined 
 (up to a permutations of its  columns) by
  the associated Milnor fiber and the marking of
  $\partial(X(C, l))$.  
\end{theorem}

\begin{proof}
  Consider the compact 4-manifold with boundary $U$, described in
  Lemma \ref{handles}. Attach it to the Milnor fiber $W(C_S, l_S)$ as
  described in Lemma \ref{attach}. This attaching is determined by the
  marking of $\partial(X(C, l))$. Then use Proposition \ref{eqmat}. 
\end{proof}

Consider now two decorated germs of plane curves. One says that they are
\emph{topologically equivalent} if one may choose Milnor
representatives of them inside euclidean balls centered at the origin
of $\mathbb{C}^2$ and homeomorphisms between the balls which restrict
to homeomorphisms between the representatives and which additionally
respect the numberings of their branches and the
decoration functions $l$. 

One may show easily that topologically equivalent
decorated germs give rise to topologically equivalent sandwiched
singularities with induced diffeomorphisms of their boundaries
preserving the associated markings. 

Look at the diffeomorphisms between the boundaries of
the two singularities which preserve the orientations and the
markings. One may consider those diffeomorphisms between fillings of their
boundaries which restrict to such special diffeomorphisms on the
boundaries. Then:

 \begin{theorem} \label{distinc}
    Consider two topologically equivalent standard decorated germs of
    plane curves, and for
    each one of them a picture deformation. If their incidence
    matrices are different up to permutation of columns, then their
    associated Milnor fibers are not diffeomorphic by an
    orientation-preserving  diffeomorphism 
    which preserves the markings of the boundaries. 
 \end{theorem}
 
 \begin{proof} This is an immediate consequence of Theorem \ref{markinc}. 
 \end{proof}

\subsection{Consequences for the existence of Stein fillings} 
  \label{conseq}

In section \ref{general}, we recalled the notion of abstract 
boundary $\partial(X,x)$ of an isolated singularity $(X,x)$ of
arbitrary dimension. In fact,
by considering the field of maximal complex subspaces of the tangent
space to a representative of the boundary, one gets a contact
structure. We denote it $(\partial(X,x), \xi(X,x))$ and we call it the
\emph{contact boundary} of $(X,x)$. It is well-defined up
to contactomorphisms which are unique up to isotopy 
(see \cite{V 80} and \cite{CNP 06}).

Any Milnor fiber of an isolated equidimensional singularity $(X,x)$
may be endowed with the structure of Stein manifold compatible with
the contact structure $\xi(X,x)$ on $\partial(X,x)$. 

By the main
theorem of \cite{CNP 06}, for normal \emph{surface} singularities the contact
structure $\xi(X,x)$ is a 
topological invariant of the germ, that is, the contact boundaries of
two singularities with homeomorphic oriented boundaries are
contactomorphic by an orientation-preserving diffeomorphism. Moreover,
as proved in \cite{CP 04}, this (unoriented) contact structure is left
invariant by all orientation-preserving diffeomorphisms of the
boundary $\partial(X,x)$, when this boundary is a rational homology sphere. 

This happens in particular for sandwiched singularities. Therefore, we
can speak about the \emph{standard contact structure} of a sandwiched
3-manifold, well-defined \emph{up to isotopy}. By
taking Milnor fibers of possibly non analytically equivalent but
topologically equivalent sandwiched singularities, one gets Stein
fillings of the same contact manifold (the boundaries being identified
by any orientation-preserving diffeomorphism). We get:

\begin{theorem} \label{lowbound}
  Let $(M, \xi)$ be a sandwiched 3-manifold endowed with its standard
  contact structure. Fix the topological type of a defining standard
  decorated germ.  Then 
  there are at least as many Stein fillings (up to diffeomorphisms
  fixed on the boundary) of $(M, \xi)$ as there are incidence matrices (up to
  permutations of columns) realised by the picture deformations of all
  the  decorated germs with the given topology. 
\end{theorem}

\subsection{An example} \label{examp}

We consider the sandwiched singularities $X(C,l)$ where $C$ is 
an ordinary 6-uple singularity and $l$ takes the constant value
$5$. The total multiplicity of each component of $C$ is equal to $1$
with respect to $C$, therefore one has to blow up 4 times on the
strict transforms of the branches of $C$, starting from their
separation produced by the blow-up of $0$. As a consequence, $(C,l)$
is a standard decorated germ and the
sandwiched singularity $X(C,l)$ has as dual graph of its minimal
resolution a star as in Figure \ref{fig:Sixuple}, where all the
vertices distinct from the central one have attached
self-intersection $-2$.  

\medskip 
\begin{figure}[h!] 
\labellist \small\hair 2pt 
\pinlabel{$-7$} at 430 123
\pinlabel{$(C,l)$} at 90 -10
\pinlabel{$X(C,l)$} at 414 -10
\pinlabel{$5$} at 178 125
\pinlabel{$5$} at 170 169
\pinlabel{$5$} at 145 199
\pinlabel{$5$} at 95 212
\pinlabel{$5$} at 35 201
\pinlabel{$5$} at 7 168
\endlabellist 
\centering 
\includegraphics[scale=0.55]{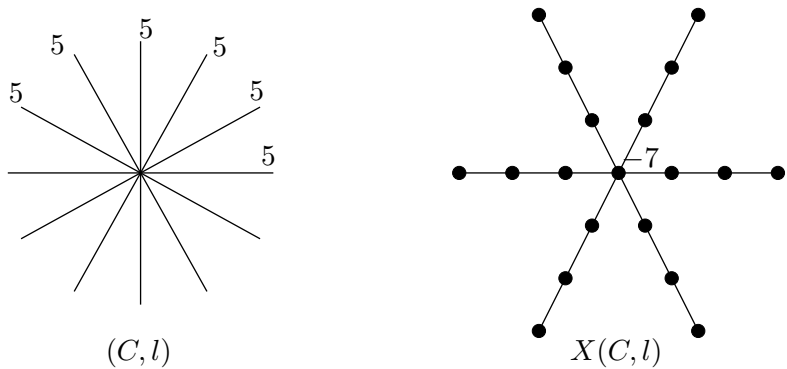} 
\vspace{2mm} \caption{A decorated ordinary germ and the
  associated sandwiched graph} 
\label{fig:Sixuple} 
\end{figure} 

In particular, one may choose for $C$ the union of six pairwise
distinct lines passing through the origin of $\C^2$. As explained in
\cite[page 501]{JS 98}, one gets \emph{in general} $323$ possibilities for
the incidence matrix of a picture deformation obtained by simply
moving the lines by parallelism. But for special positions of the 
lines, one gets \emph{one more} incidence matrix, the configuration
generating it being drawn in Figure \ref{fig:Speconf}. By Theorem
\ref{lowbound}, we see that the contact 
boundary of such sandwiched singularities has at least $324$ Stein
fillings, up to orientation-preserving diffeomorphisms fixed on the
boundary. We expect that these are \emph{all} the Stein fillings of that
contact manifold, up to strong diffeomorphisms.

\medskip 
\begin{figure}[h!] 
\centering 
\includegraphics[scale=0.55]{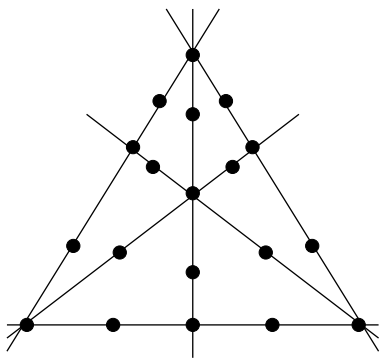} 
\vspace{2mm} \caption{The special picture deformation} 
\label{fig:Speconf} 
\end{figure}

\end{document}